\documentclass[11pt]{amsart}
\usepackage{amssymb}
\usepackage[margin=1in]{geometry}
\usepackage[colorlinks]{hyperref}
\newtheorem{theorem}{Theorem}[section]
\newtheorem{lem}[theorem]{Lemma}
\newtheorem{prop}[theorem]{Proposition}
\newtheorem{cor}[theorem]{Corollary}

\theoremstyle{definition}
\newtheorem{definition}[theorem]{Definition}
\newtheorem{example}[theorem]{Example}

\theoremstyle{remark}
\newtheorem{remark}[theorem]{Remark}

\numberwithin{equation}{section}

\begin{document}

\newcommand{\spacing}[1]{\renewcommand{\baselinestretch}{#1}\large\normalsize}
\spacing{1.14}

\title{Left Invariant Randers Metrics of Berwald type on Tangent Lie Groups}

\author {F. Asgari}

\address{Department of Mathematics\\ Faculty of  Sciences\\ University of Isfahan\\ Isfahan\\ 81746-73441-Iran.} \email{farhad\_13812003@yahoo.com}

\author {H. R. Salimi Moghaddam}

\address{Department of Mathematics\\ Faculty of  Sciences\\ University of Isfahan\\ Isfahan\\ 81746-73441-Iran.} \email{hr.salimi@sci.ui.ac.ir and salimi.moghaddam@gmail.com}



\begin{abstract}
Let $G$ be a Lie group equipped with a left invariant Randers metric of Berward type $F$, with underlying left invariant Riemannian metric
$g$. Suppose that $\widetilde{F}$ and $\widetilde{g}$ are lifted Randers and Riemannian metrics arising from $F$ and $g$ on the tangent Lie group $TG$ by vertical and complete lifts. In this article we study the relations between the flag curvature of the Randers manifold $(TG,\widetilde{F})$ and the sectional curvature of the Riemannian manifold $(G,g)$ when $\widetilde{F}$ is of Berwald type. Then we give all simply connected $3$-dimentional Lie groups such that their tangent bundles admit Randers metrics of Berwarld type and their geodesics vectors.
\end{abstract}

\maketitle


\section{\textbf{Introduction}}
Finsler geometry is one of the important subjects in differential geometry
which has been developed in the last century. One of the important reasons for this development is its application in many areas of
natural science such as physics and biology (for more details see \cite{AnInMa} and \cite{Asanov}). An important type of Finsler metrics introduced in general relativity, because of its application, is the family of Randers metrics (see \cite{Randers}). Also, in four-dimensional case, Randers metrics have been used for computing the Lagrangian function of a test electric charge in the electromagnetic and gravitational fields (see \cite{Asanov}, \cite{Ingarden} and \cite{Landau-Lifshitz}).\\
Among the Finsler spaces, the family of Lie groups equipped with invariant Finsler metrics are of special interest to geometers. In the last decade, many geometric properties of such spaces have been studied (for example see \cite{Deng1,Deng2}, \cite{Latifi1,Latifi2} and \cite{Salimi1,Salimi2,Salimi3}).
If $G$ is a Lie group we can easily see its tangent bundle is also a Lie group (see \cite{Hilgert-Neeb}). In this article we study some special types of left invariant Randers metrics on the Lie group $TG$.\\
The study of the Riemannian geometry of tangent bundles started with Sasaki's paper \cite{Sasaki}. He showed that any Riemannian metric $g$ on the base manifold $M$ induces a Riemannian metric on $TM$ by using vertical and horizontal lifts. If we replace the horizontal lift with complete lift then we have another way for constructing Riemannian metrics on $TM$. K. Yano and S. Kobayashi used this way and studied many geometric properties of such lifted metrics (see \cite{YaKo1, YaKo2, YaKo3}). For more results in this field one can see \cite{KowalskiSekizawa}, \cite{Sekizawa} and \cite{YaIs}.\\
In our previous two papers \cite{Asgari-Salimi1,Asgari-Salimi2}, we study the Riemannian geometry of invariant Riemannian metrics induced on $TG$ by using vertical and complete lifts. In this work we study lifted Randers metrics of Berwald type on $TG$. In \cite{Aradi}, it has been shown that any left invariant Finsler metric on a Lie group is a generalized Berwald metric so the Berwaldian condition is not very restrictive. We study the curvature properties of such spaces, then we give all simply connected $3$-dimensional Lie groups such that their tangent bundles admit Randers metrics of Berwarld type.

\section{\textbf{Preliminaries}}
Suppose that $M$ is a $m-$dimensional smooth manifold and $TM$ is its tangent bundle. Let $X$ be an arbitrary vector field on $M$. Then $X$ defines two types of (local) one-parameter group of diffeomorphisms on $TM$ as follows,
\begin{eqnarray}\label{one-parameter group of diffeomorphisms}
  \Phi_t(y)&:=&(T_x\phi_t)(y), \hspace*{3cm} \forall x\in M, \forall y\in T_xM.\\
  \Psi_t(y)&:=&y+tX(x),
\end{eqnarray}
where $\phi_t$ is the flow generated by the vector field $X$ on $M$. The infinitesimal generator of one parameter groups of diffeomorphisms $\Phi_t$ and $\Psi_t$ are called the complete lift (denoted by $X^c$) and vertical lift (denoted by $X^v$) of $X$ respectively. \\
Let $(x^i) (i = 1,2, \cdots , n)$ be a local coordinate system in an open subset $U$ of $M$. Then we denote the induced local coordinate system on $\pi^{-1}(U)$ by $(x^i,y^i) (i = 1,2, \cdots , n)$, where $\pi:TM\longrightarrow M$ is the projection map. Suppose that $X$ is a vector field on $M$ with local representation $X|_U=\sum_{i=1}^m\xi^i\frac{\partial}{\partial x^i}$. Then, the local representation of its vertical and complete lifts on $TM$ are as follows:
\begin{eqnarray}
  (X|_U)^v &=& \sum_{i=1}^m\xi^i\frac{\partial}{\partial y^i}, \\
  (X|_U)^c &=& \sum_{i=1}^m\xi^i\frac{\partial}{\partial x^i}+\sum_{i=1}^m\sum_{j=1}^m\frac{\partial\xi^i}{\partial x^j}y^j\frac{\partial}{\partial y^i}.
\end{eqnarray}
The Lie brackets of vertical and complete lifts of vector fields satisfy the following equations (for more details see \cite{Hind} and \cite{YaIs}),

\begin{eqnarray}\label{Lie bracket}
  \left[X^v, Y^v\right] &=& 0, \nonumber\\
  \left[X^c, Y^c\right] &=& \left[X, Y\right]^c, \\
  \left[X^v, Y^c\right] &=& \left[X, Y\right]^v.\nonumber
\end{eqnarray}
Now let $G$ be a real $m-$dimensional connected Lie group with multiplication map $\mu:G\times G\longrightarrow G$, $(x,y)\longrightarrow xy$
and identity element $e$. Suppose that $l_y:G \longrightarrow G$, $x\longrightarrow yx$  and $r_y:G\longrightarrow G$, $x\longrightarrow xy$ are left and right translations, respectively. Then $TG$ is also a Lie group with multiplication:
\begin{eqnarray}
       T\mu :(v,w)\longrightarrow T_yl_xw+T_xr_yv, \hspace{2cm} v\in T_xG, w\in T_yG
\end{eqnarray}
with identity element $0_e\in T_eG$ and the inversion map $T\iota$, where $\iota$ is the inversion map of $G$ (see \cite{Hilgert-Neeb}).\\
In \cite{Hind}, it is shown that if $X$ is a left invariant vector field on $G$ then $X^c$ and $X^v$ are left invariant vector fields on $TG$. This result together with  the local representation of vertical and complete lifts of vector fields show that if $\{X_1,X_2,...,X_m\}$ is a basis for the Lie algebra $\mathfrak{g}$ of $G$, then $\{X_1^c,...,X_n^c,X_1^v,...,X_m^v\}$ is a basis for the Lie algebra $\widetilde{\mathfrak{g}}$ of $TG$.\\

Now we give some preliminaries from Finsler geometry (see \cite{Bao-Chern-Shen}).
\begin{definition}
A Finsler metric  $F$  on $M$ is a non-negative function
$F:TM \longrightarrow R$ with the following properties:
\begin{itemize}
  \item $F$ is smooth on the slit tangent bundle $TM^0=TM\backslash0$,
  \item $F(x,\lambda y)=\lambda F(x,y)$, for any $x\in M, y\in T_xM$ and $\lambda>0$,
  \item The Hessian matrix $g_{ij}(x,y)=\frac{1}{2}\frac{\partial^2F^2}{\partial y^i\partial y^j}$ is positive definite at every point $(x,y)\in TM^0$.
\end{itemize}
\end{definition}

As we mentioned in introduction, an important class of Finsler metrics is the class of Randers metrics of the form,
\begin{eqnarray}
F(x,y)=\sqrt{g_x(y,y)}+b_i(x)y^i, \hspace{2cm} y\in T_xM,
\end{eqnarray}
where $g$ is a Riemennian metric and $b=b_i(x)$ is a 1-form on $M$ and $\|b\|_g=\sqrt{b_i(x)b^i(x)}<1$ where $b^i(x):=g^{ij}(x)b_j(x)$.

On Lie groups and also homogeneous spaces, it is more convenient to consider a Randers metric as follows (see \cite{Deng1} and \cite{Salimi1}),
\begin{eqnarray}\label{invariant Randers metric formula}
F(x,y)=\sqrt{g_x(y,y)}+g(X(x),y),
\end{eqnarray}
where $X$ is a vector field on $M$ with $\|X\|_g:=\sqrt{g(X,X)}<1$. \\
A Finsler metric $F$ on a Lie group $G$ is called left invariant if,
\begin{eqnarray}
F(x,y)=F(e,T_xl_{x^{-1}}y), \hspace*{2cm}\forall x\in G, \forall y\in T_xG.
\end{eqnarray}
So if $g$ is a left invariant Riemannian metric and $X$ is a left invariant vector field on $G$ with $\|X\|:=g(X,X)<1$, then the Randers metric $F$ with underlying Riemannian metric $g$ and vector field $X$ is a left invariant Finsler metric.

An important quantity in Finsler geometry is the concept of flag curvature with the following formula:
\begin{eqnarray}
K(P,Y)=\frac{g_y(R(u,y)y,u)}{g_y(y,y)g_y(u,u)-g^2_y(y,u)},
\end{eqnarray}
where $g_y(u,v)=\frac{1}{2}\frac{\partial^2}{\partial s \partial t}(F^2(y+su+tv)|_{s=t=0}$ is the fundumental tensor , $P=span\{y,u\}$ is the plan spanned by vectors $y$ and $u$, $R(u,y)y=\nabla_u\nabla_yy-\nabla_y\nabla_uy-\nabla_{[u,y]}y$ is the curvature tensor and $\nabla$ is the chern connection of $F$.\\
In special case if $X$ is parallel with respect to the Levi-Civita connection induced by the Riemannian metric $g$, then $F$ is called of Berwald type. If $F$ is of Berwald type then the Chern connection of $F$ and the Levi-Civita connection of $g$ are coincide (see \cite{Bao-Chern-Shen}).

\section{\textbf{Lifting of Randers metrics on tangent bundles}}

Let $G$ be a Lie group equipped with a left invariant Riemannian metric $g$. In \cite{Asgari-Salimi1}, we defined a natural left invariant metric $\widetilde {g}$ on $TG$ as follows,
\begin{eqnarray}\label{Lift}
\widetilde {g} (X^c,Y^c)&=&g(X,Y), \nonumber \\
\widetilde {g} (X^v,Y^v)&=&g(X,Y), \\
\widetilde {g} (X^c,Y^v)&=&0,\nonumber
\end{eqnarray}
where $X$ and $Y$ are arbitrary vector fields on $G$.
Also we proved the following proposition about its Levi-Civita connection.
\begin{prop}\label{Levi-Civita connection}
If $\widetilde {\nabla}$ is the Levi-Civita connection induced by $\widetilde{g}$ on $TG$ then,
\begin{eqnarray}
\widetilde {\nabla}_{X^c}Y^c&=&(\nabla_XY)^c, \nonumber\\
\widetilde {\nabla}_{X^v}Y^v&=&(\nabla_XY-\frac{1}{2}[X,Y])^c, \\
\widetilde {\nabla}_{X^c}Y^v&=&(\nabla_XY+\frac{1}{2}ad^*_YX)^v, \nonumber\\
\widetilde {\nabla}_{X^v}Y^c&=&(\nabla_XY+\frac{1}{2}ad^*_XY)^v,\nonumber
\end{eqnarray}
where $\nabla$ is the Levi-Civita connection of $g$ and $X, Y$ are any two left invariant vector fields on $G$.
\end{prop}

Let $F$ be a left invariant Randers metric on $G$ defined by \ref{invariant Randers metric formula}, where $g$ and $X$ are
a left invariant Riemannian metric and a left invariant vector field on $G$, respectively. Now, a natural question is: How can we lift such a Randers metric on $TG$ such that the lifted metric be left invariant? \\
The answer is very simple because we have the following Randers metrics,
\begin{eqnarray}
  F^c((x,y),\widetilde{z})&=&\sqrt{\widetilde{g}(\widetilde{z},\widetilde{z})}+\widetilde{g}(X^c(x,y),\widetilde{z}),\label{F^c} \\
  F^v((x,y),\widetilde{z})&=&\sqrt{\widetilde{g}(\widetilde{z},\widetilde{z})}+\widetilde{g}(X^v(x,y),\widetilde{z}),\label{F^v}
\end{eqnarray}
where $x\in G$, $y\in T_xG$ and $\widetilde{z}\in T_{(x,y)}TG$.\\
We can easily see these are left invariant Randers metrics on $TG$ because
\begin{equation}
    \|X^c\|_{\widetilde{g}}=\|X^v\|_{\widetilde{g}}=\|X\|_g<1.
\end{equation}
From now on we suppose that $F$, $F^c$ and $F^v$ are defined as above.\\
The following lemmas give a necessary and sufficient condition for $F^c$ and $F^v$ to be of Berwarld type.

\begin{lem}\label{complete lift}
Suppose that $F$ is an arbitrary left invariant Randers metric defined by \ref{invariant Randers metric formula}, where $g$ is a left invariant Riemannian metric and $X$ is a left invariant vector field on a Lie group $G$. Then, $F^c$ is of Berwald type if and only if $F$ is of Berwald type.
\end{lem}
\begin{proof}
Let $F^c$ be of Berwald type, so for any $Y\in\frak{g}$ we have $\widetilde {\nabla}_{Y^c}X^c=\widetilde {\nabla}_{Y^v}X^c=0$. Now, the proposition \ref{Levi-Civita connection} shows that for any $Y\in\frak{g}$ we have $\nabla_YX=0$ which means that $F$ is of Berwald type. Now let $F$ be of Berwald type. By using theorem 3.1 of \cite{Deng3} we have
\begin{equation*}
    g(ad_Y^\ast X,Z)=g(X,[Y,Z])=0, \ \ \ \ \ \forall Y,Z\in\frak{g}.
\end{equation*}
The last relation shows that, for any $Y\in\frak{g}$, $ad_Y^{\ast}X=0$. Now proposition \ref{Levi-Civita connection} proves that $\widetilde {\nabla}_{Y^c}X^c=\widetilde {\nabla}_{Y^v}X^c=0$, which means that $F^c$ is of Berwald type.
\end{proof}

\begin{lem}\label{vertical lift}
With the assumptions of the previous lemma, $F^v$ is of Berwald type if and only if $ad^\ast_X=ad_X$ and $\nabla_XY=\frac{1}{2}[X,Y]$.
\end{lem}
\begin{proof}
$F^v$ is of Berwald type if and only if $\widetilde{\nabla}_{Y^c}X^v=\widetilde {\nabla}_{Y^v}X^v=0$. Now it is sufficient to use proposition \ref{Levi-Civita connection}.
\end{proof}
The following corollary is an immediate consequence of the previous lemma and proposition \ref{Levi-Civita connection}.
\begin{cor}\label{vertical lift with Berwald condition}
In the previous lemma if we consider $F$ is of Berwald type then, $F^v$ is of Berwald type if and only if $X\in z(\frak{g})$, where $z(\frak{g})$ denotes the center of $\frak{g}$.
\end{cor}

In the following two theorems we give the flag curvature formulas of $F^c$ and $F^v$ explicitly, where $F$ is of Berwald type.

\begin{theorem}\label{flag curvature F^c}
Let $G$ be a Lie group equipped with a left invariant Riemannian metric $g$. Suppose that $F(x,y)=\sqrt{g(y,y)}+g(X(x),y)$ is a Randers metric of Berwald type on $G$. Then for the flag curvature of the left invariant Randers metric $F^c$ on $TG$ we have:
\begin{equation*}
\begin{split}
  & (1) \ \widetilde{P}=span\{Y^c,U^c\}, K^{F^c}(\widetilde{P},Y^c)=\frac{K(U,Y)}{(1+g(X,Y))^2},\\
  & (2) \ \widetilde{P}=span\{Y^c,U^v\}, K^{F^c}(\widetilde{P},Y^c)=\frac{1}{(1+g(X,Y))^2}\{K(U,Y)+\frac{1}{2}g([U,\nabla_YU],Y)\\
  &\hspace*{7cm}-\frac{1}{2}g(\nabla_U ad_U^\ast Y,Y)+\frac{1}{4}g([U,ad_U^\ast Y],Y)\\
  &\hspace*{7cm}-\frac{1}{2}g([[Y,U],U],Y)\},\\
  & (3) \ \widetilde{P}=span\{Y^v,U^c\}, K^{F^c}(\widetilde{P},Y^v)=K(U,Y)+\frac{1}{2}g([Y,\nabla_UY],U)\\
  &\hspace*{7cm}-\frac{1}{2}g(\nabla_Y ad_Y^\ast U,U)+\frac{1}{4}g([Y,ad_Y^\ast U],U)\\
  &\hspace*{7cm}-\frac{1}{2}g([[U,Y],Y],U),\\
  & (4) \ \widetilde{P}=span\{Y^v,U^v\}, K^{F^c}(\widetilde{P},Y^v)=K(U,Y)+g(\nabla_{[U,Y]}Y,U)+\frac{1}{4}\|[U,Y]\|^2,
\end{split}
\end{equation*}
where $K^{F^c}$ and $K$ denote the flag curvature of $F^c$ and the sectional curvature of $g$ respectively, and $\{Y,U\}$ is an orthonormal basis for a subspace $P$ of $\frak{g}$, with respect to $g$.
\end{theorem}
\begin{proof}
Lemma \ref{complete lift} shows that $F^c$ is of Berwald type, thus the Chern connection of $F^c$ and the Levi-Civita connection of $\widetilde{g}$ coincide. On the other hand, we can easily see,

\begin{equation*}
    \frac{\widetilde{g}(Y^c,Y^c)}{F^c(Y^c)^2}=\frac{1}{(1+g(X,Y))^2} \ \ \ \ , \ \ \ \ \frac{\widetilde{g}(Y^v,Y^v)}{F^c(Y^v)^2}=1.
\end{equation*}
Now theorem 2.4 of \cite{Asgari-Salimi1} together with the curvature formula of \cite{Deng4} complete the proof.
\end{proof}

\begin{theorem}\label{flag curvature F^v}
Suppose that $G$ is a Lie group equipped with a left invariant Riemannian metric $g$. Assume that $F(x,y)=\sqrt{g(y,y)}+g(X(x),y)$ is a left invariant Randers metric on $G$ such that $F^v$ is of Berwald type. Then for the flag curvature of the left invariant Randers metric $F^v$ on $TG$ we have:
\begin{equation*}
\begin{split}
  & (1) \ \widetilde{P}=span\{Y^c,U^c\}, K^{F^v}(\widetilde{P},Y^c)=K(U,Y),\\
  & (2) \ \widetilde{P}=span\{Y^c,U^v\}, K^{F^v}(\widetilde{P},Y^c)=K(Y,U)+\frac{1}{2}g([U,\nabla_YU],Y)\\
  &\hspace*{7cm}-\frac{1}{2}g(\nabla_U ad_U^\ast Y,Y)+\frac{1}{4}g([U,ad_U^\ast Y],Y)\\
  &\hspace*{7cm}-\frac{1}{2}g([[Y,U],U],Y)\},\\
  & (3) \ \widetilde{P}=span\{Y^v,U^c\}, K^{F^v}(\widetilde{P},Y^v)=\frac{1}{(1+g(X,Y))^2}\{K(U,Y)+\frac{1}{2}g([Y,\nabla_UY],U)\\
  &\hspace*{7cm}-\frac{1}{2}g(\nabla_Y ad_Y^\ast U,U)+\frac{1}{4}g([Y,ad_Y^\ast U],U)\\
  &\hspace*{7cm}-\frac{1}{2}g([[U,Y],Y],U)\},\\
  & (4) \ \widetilde{P}=span\{Y^v,U^v\}, K^{F^v}(\widetilde{P},Y^v)=\frac{1}{(1+g(X,Y))^2}\{K(U,Y)+g(\nabla_{[U,Y]}Y,U)+\frac{1}{4}\|[U,Y]\|^2\},
\end{split}
\end{equation*}
where $K^{F^v}$ denotes the flag curvature of $F^v$, and the other notations are as the previous theorem.
\end{theorem}
\begin{proof}
It is sufficient to note that
\begin{equation*}
    \frac{\widetilde{g}(Y^v,Y^v)}{F^v(Y^v)^2}=\frac{1}{(1+g(X,Y))^2} \ \ \ \ , \ \ \ \ \frac{\widetilde{g}(Y^c,Y^c)}{F^v(Y^c)^2}=1.
\end{equation*}
The other parts of proof are similar to the previous theorem so we omit it.
\end{proof}
\begin{example}\label{five-dimensional}
As an example we can consider simply connected two-steps nilpotent Lie groups of dimension five with three-dimensional center. In \cite{Salimi3}, the second author has shown that a simply connected two-steps nilpotent Lie group admits a left invariant Randers metric of Berwald type if and only if it has three-dimensional center. In this case, it is shown that $X\in z(\frak{g})$. Therefore, if $G$ is a simply connected two-steps nilpotent Lie group of dimension five with three-dimensional center then by using lemma \ref{complete lift} and corollary \ref{vertical lift with Berwald condition}, the ten-dimensional Lie group $TG$ admits left invariant Randers metrics $F^v$ and $F^c$ of Berwald type.
\end{example}
\begin{remark}
By using formula 5.3 of \cite{Salimi3} and the above theorems, we can see the flag curvatures of the Finsler metrics $F^v$ and $F^c$ on the Lie group $TG$ of example \ref{five-dimensional} admit negative, positive and zero values.
\end{remark}


\section{\textbf{Randers metrics of Berwald type on the tangent Bundles of 3-dimensional lie groups}}

In \cite{Salimi2}, the second author have shown that the only simply connected three dimensional real Lie groups which admit Randers metrics of Berwald type, are as table \ref{three dimensional Lie groups}.

\fontsize{8}{0}{\selectfont
\begin{table}[h!]
    \centering\caption{}\label{three dimensional Lie groups}
        \begin{tabular}{|p{0.5cm}|p{2cm}|p{1.7cm}|p{2cm}|p{1.3cm}|p{3cm}|p{2.5cm}|}
        \hline
            case & Lie algebra structure & Associated  simply connected  Lie group & Left invariant Riemannian metric $g$ & Conditions for  parameters of $g$ & $X$, where $F(x,y)=\sqrt{g_x(y,y)}+g(X(x),y)$ & Conditions for  parameters of $X$ \\
             \hline
           case 1 & $[W,Y]=0,$ $[W,Z]=0,$ $[Y,Z]=0$ & $\Bbb{R}^3$ & $\left(
                                                                         \begin{array}{ccc}
                                                                           1 & 0 & 0 \\
                                                                           0 & 1 & 0 \\
                                                                           0 & 0 & 1 \\
                                                                         \end{array}
                                                                       \right)$ & - & $X=pW+qY+rZ$ &  $\sqrt{p^2+q^2+r^2}<1$ \\
            \hline
           case 2 & $[W,Y]=0,$ $[W,Z]=-Y,$ $[Y,Z]=-2Y$ & The nonunimodular group $G_0$ & $\left(
                                                                         \begin{array}{ccc}
                                                                           1 & \frac{1}{2} & 0 \\
                                                                           \frac{1}{2} & 1 & 0 \\
                                                                           0 & 0 & \nu \\
                                                                         \end{array}
                                                                       \right)$ & $\nu>0$ & $X=-2pW+pY$ &  $|p|<\frac{\sqrt{3}}{3}$ \\
            \hline
           case 3 & $[W,Y]=0,$ $[W,Z]=Y,$ $[Y,Z]=-W$ & The solvable Lie group $\widetilde{E_0}(2)$ & $\left(
                                                                         \begin{array}{ccc}
                                                                           1 & 0 & 0 \\
                                                                           0 & 1 & 0 \\
                                                                           0 & 0 & \nu \\
                                                                         \end{array}
                                                                       \right)$ & $\nu>0$ & $X=pZ$ &  $|p|<\frac{1}{\sqrt{\nu}}$ \\
            \hline
        \end{tabular}
\end{table}
}
\fontsize{11}{0}
\large\normalsize
\spacing{1.14}

Now we study the existence of left invariant Berwaldian Randers metrics of the forms $F^c$ and $F^v$ on the tangent bundle of simply connected three dimensional real Lie groups.
\begin{theorem}
Suppose that $F(x,y)=\sqrt{g_x(y,y)}+g(X(x),y)$ is any left invariant Randers metric on an arbitrary simply connected three-dimensional real Lie group $G$. Then the six-dimensional Lie group $TG$ admits a left invariant Randers metric of Berwald type of the form $F^c$ if and only if $G$ is one of the Lie groups described in table \ref{three dimensional Lie groups}.
\end{theorem}
\begin{proof}
It is sufficient to use theorem 2.2 of \cite{Salimi2} and lemma \ref{complete lift}.
\end{proof}

\begin{theorem}
Let $F(x,y)=\sqrt{g_x(y,y)}+g(X(x),y)$ be an arbitrary left invariant Randers metric of Berwald type on any simply connected three-dimensional real Lie group $G$. Then the six-dimensional Lie group $TG$ admits a left invariant Randers metric of Berwald type of the form $F^v$ if and only if $G$ is one of the cases 1 and 2 described in table \ref{three dimensional Lie groups}.
\end{theorem}
\begin{proof}
Theorem 2.2 of \cite{Salimi2} together with corollary \ref{vertical lift with Berwald condition} complete the proof.
\end{proof}
Now we give some results about geodesic vectors.
\begin{theorem}
Let $g$ be a left invariant Riemannian metric on an arbitrary Lie group and $\frak{g}$ be the Lie algebra of $G$. Suppose that $F$ is a left invariant Randers metric of Douglas type defined by $g$ and a left invariant vector field $X$. Then, $U\in\frak{g}\backslash\{0\}$ is a geodesic vector of $(G,F)$ if and only if $U$ is a geodesic vector of $(G,g)$.
\end{theorem}
\begin{proof}
If we consider a Randers metric as a $(\alpha,\beta)$-metric, we have $\phi(s)=1+s$. So, for any Randers metric we have $\phi''(s)=0$, for all $s\in\Bbb{R}$. On the other hand, since we have considered $F$ is of Douglas type, we have $g(X,[Y,Z])=0$, for any $Y,Z\in\frak{g}$ (see \cite{Deng4}). Therefore, $F$ satisfies in the conditions of theorem 2.3 of \cite{Parhizkar-Salimi}. The last expression competes the proof.
\end{proof}
\begin{example}
The previous theorem shows that for geodesic vectors of Randers metrics of non-trivial cases 2 and 3 of table \ref{three dimensional Lie groups} we have:
\begin{description}
  \item[case 2] $U$ is a geodesic vector if and only if $U=aW-\frac{a}{2}Y+cZ$ or $U=aW+\frac{a}{2}Y$, where $a, c \in\Bbb{R}$.
  \item[case 3] $U$ is a geodesic vector if and only if $U=cZ$ or $U=aW+bY$, where $a, b, c \in\Bbb{R}$.
\end{description}
\end{example}
\begin{cor}
By using the previous theorem we can easily see in the cases of lemmas \ref{complete lift} and \ref{vertical lift}, $U^c$ (or $U^v$) is a geodesic vector if and only if $U$ is a geodesic vector of $(G,g)$.
\end{cor}

{\large{\textbf{Acknowledgment.}}} We are grateful to the office of Graduate Studies of the University of Isfahan for their support. This research was supported by the Center of Excellence for
Mathematics at the University of Isfahan.


\bibliographystyle{amsplain}

\end{document}